\documentclass[hidelinks]{article}
\usepackage{graphicx} 
\usepackage{hyperref}
\usepackage{amsthm}
\usepackage{thm-restate}
\usepackage{geometry}
\usepackage{amssymb}
\usepackage{amsmath}
\usepackage{cleveref}
\usepackage{comment}
\usepackage{xcolor}
\usepackage[shortlabels]{enumitem}
\usepackage{listings}
\usepackage{blkarray}
\usepackage{lmodern}
\usepackage[english]{babel}
\usepackage{float}
\usepackage{tikz}

\newtheorem{theorem}{Theorem}[section]
\newtheorem{lemma}[theorem]{Lemma}
\newtheorem{cor}[theorem]{Corollary}

\newtheorem{observation}[theorem]{Observation}

\newtheorem{proposition}[theorem]{Proposition}

\theoremstyle{definition}

\newcommand{\eps}{\varepsilon}

\title{A remark on the independence number of sparse random Cayley sum graphs}

\author{
Rajko Nenadov\thanks{School of Computer Science, University of Auckland, New Zealand. Email: \texttt{rajko.nenadov@auckland.ac.nz}. Research supported by the Marsden Fund of the Royal Society of New Zealand.}
}
\date{}

\begin{document}

\maketitle

\begin{abstract}
The Cayley sum graph $\Gamma_S$ of a set $S \subseteq \mathbb{Z}_n$ is defined on the vertex set $\mathbb{Z}_n$, with an edge between distinct $x, y \in \mathbb{Z}_n$ if $x + y \in S$.  Campos, Dahia, and Marciano have recently shown that if $S$ is formed by taking each element in $\mathbb{Z}_n$ independently with probability $p$, for $p > (\log n)^{-1/80}$, then with high probability the largest independent set in $\Gamma_S$ is of size 
$$
    (2 + o(1)) \log_{1/(1-p)}(n).
$$
This extends a result of Green and Morris, who considered the case $p = 1/2$, and asymptotically matches the independence number of the binomial random graph $G(n,p)$. We improve the range of $p$ for which this holds to $p > (\log n)^{-1/3 + o(1)}$. The heavy lifting has been done by Campos, Dahia, and Marciano, and we show that their key lemma can be used a bit more efficiently.
\end{abstract}

\section{Introduction}

The Cayley sum graph $\Gamma_S$ of a set $S \subseteq \mathbb{Z}_n$ is defined to have the vertex set $\mathbb{Z}_n$ and an edge between two distinct vertices $x, y \in \mathbb{Z}_n$ if $x + y \in S$. We are interested in properties of $\Gamma_S$ when $S$ is a random subset. To this end, for $p \in (0,1)$, we denote with $\Gamma_p$ the Cayley sum graph $\Gamma_S$ where $S \subseteq \mathbb{Z}_n$ is formed by taking each element in $\mathbb{Z}_n$ with probability $p$, independently of all other elements. We call such $S$ a \emph{$p$-random} subset. Starting with Green \cite{gree05clique}, the question of which parameters of $\Gamma_p$ match those of the binomial random graph $G(n,p)$ has attracted significant attention.

Given a graph $G$, we let $\alpha(G)$ denote the size of a largest independent set in $G$ -- a quantity commonly referred to as the \emph{independence number}. Green \cite{gree05clique} showed that the independence number $\alpha(\Gamma_{1/2})$ is at most $160 \log n$ with high probability. This was further strengthened by Green and Morris \cite{green16counting} to $(2 + o(1)) \log(n)$, which matches the bound on the independence number of $G(n,1/2)$. Campos, Dahia, and Marciano \cite{campos24cayley} have recently showed that $\alpha(\Gamma_p) \approx \alpha(G(n,p))$ for $p > (\log n)^{-1/80}$. Here we further extend the range of $p$.

\begin{theorem} \label{thm:main}
Let $n$ be a prime number and let $p = p(n)$ satisfy $p > (\log n)^{-1/3 + o(1)}$. Then 
$$
    \alpha(\Gamma_p) = (2 + o(1)) \log_{1/(1-p)}(n)
$$
with high probability as $n \rightarrow \infty$.
\end{theorem}

The overall proof is largely the same as the one of Campos, Dahia, and Marciano \cite{campos24cayley}, with \cite[Theorem 2.1]{campos24cayley} being replaced by Lemma \ref{lemma:fingerprint} (the \emph{Fingerprint lemma}), and \cite[Corollary 8.4]{campos24cayley} being replaced by Lemma \ref{lemma:freiman_cyclic} (the Freiman-Ruzsa theorem for cyclic groups of prime order). Of the two, Lemma \ref{lemma:fingerprint} is more interesting from the point of view of novelty. The proof of Lemma \ref{lemma:freiman_cyclic} follows parts the exposition of Green \cite{green_notes} based on the work of Chang \cite{chang02polynomial}.



\section{Preliminaries}

Throughout the paper we tacitly avoid explicit use of floor and ceiling. 

\paragraph{Generalised arithmetic progression.} For simplicity, we only consider shifted centred progressions. Namely, a \emph{$d$-dimensional generalised arithmetic progression}, or \emph{$d$-GAP} for short, in an abelian group $G$ is a set $P$ of the form
$$
    \left\{ v_0 + \sum_{i = 1}^d n_i v_i \mid -N_i \le n_i \le N_i \right\},
$$
for some $v_0, v_1, \ldots, v_d \in G$ and $N_1, \ldots, N_d \in \mathbb{N}$. The size of $P$ is defined as $\mathrm{size}(P) = \prod_{i = 1}^r (2N_i + 1)$. Note that $\textrm{size}(P)$ is an upper bound on the cardinality of $P$. If $v_0 = 0$, we say that $P$ is centred.

\paragraph{Freiman isomorphism.} Given additive sets $X$ and $Y$, a bijective mapping $\phi: X \mapsto Y$ is called a \emph{Freiman isomorphism} if $a + b = a' + b'$ if and only if $\phi(a) + \phi(b) = \phi(a') + \phi(b')$, whenever $a, b, a', b' \in X$. 

The \emph{Freiman dimension} of a set $A$, denoted by $\dim_f(A)$, is defined as the largest $d \in \mathbb{N}$ for which there is a subset of $\mathbb{Z}^d$ of full rank which is Freiman isomorphic to $A$. As observed in \cite{campos24cayley}, the following is a simple corollary of Freiman's lemma \cite{freiman99lemma} on sumsets of sets with full rank in $\mathbb{R}^d$.

\begin{observation} \label{obs:upper_bound}
    For $X \subseteq \mathbb{Z}_n$ with $|X+X| \le K|X|$ we have $\dim_f(X) < 2K$.
\end{observation}

\section{The Fingerprint lemma}

The following lemma is the key new ingredient in the proof of Theorem \ref{thm:main}. 

\begin{lemma} \label{lemma:fingerprint}
    For every $\alpha > 0$ there exists $\xi, C > 0$ such that the following holds. Suppose $A \subseteq \mathbb{Z}_n$ is such that $|A + A| \le K|A|$ for some $K \le \xi |A| / \log^2(|A|)$. Then there exist $F \subseteq X \subseteq A$ such that     
    $$
        |X| \ge (1 - \alpha)|A|, \; |F| = C \sqrt{\dim_f(X) |A|} \quad \text{and} \quad
        |F + F| \ge (1 - \alpha) \frac{(\dim_f(X) + 1)|A|}{2}.
    $$    
\end{lemma}

The key notion in the proof of Lemma \ref{lemma:fingerprint} is that of a robust Freiman dimension, introduced in \cite{campos24cayley}. Here we state it slightly modified:  A set $X$ is \emph{$(\eps, \beta)$-Freiman-robust} if $\textrm{dim}_f(X') \ge (1 - \beta)\dim_f(X)$ for every $X' \subseteq X$ of size $|X'| \ge (1 - \eps)|X|$. The following is analogous to \cite[Proposition 8.3]{campos24cayley}.

\begin{lemma} \label{lemma:large_robust}
    Let $A \subseteq \mathbb{Z}_n$ with $|A + A| \le K|A|$ and let $\eps < 1/L$, where $L = \log_{1/(1-\beta)}(2K)$ for some $\beta > 0$. Then there exists $X \subseteq A$ such that $|X| > (1 - \eps L)|A|$ and $X$ is $(\eps, \beta)$-Freiman-robust.
\end{lemma}
\begin{proof}
    Set $X = A$ and repeat the following: If there exists $X' \subseteq X$ of size $|X'| \ge (1 - \eps)|X|$ such that $\dim_f(X') < (1 - \beta)\dim_f(X)$, set $X := X'$; otherwise terminate. As $\dim_f(A) < 2K$ (Observation \ref{obs:upper_bound}) and trivially $\dim_f(X) \ge 1$ for every non-empty $X$, the process terminates after $t \le L := \log_{1/(1-\beta)} (2K)$ rounds. The size of the resulting set $X$ is at least 
    $$
        |X| \ge (1 - \eps)^t|A| \ge (1 - t \eps)|A|.
    $$
    By the definition of the process, $X$ is $(\eps, \beta)$-Freiman-robust.
\end{proof}

The proof of Lemma \ref{lemma:fingerprint} builds on iterated application of the following two results. Similarly as in the proof of \cite[Theorem 2.2]{campos24cayley}, we employ the first one when $d = O(1)$ and the second when $d$ is large.

\begin{theorem}[{\cite[Theorem 1.2]{jing2023kemperman}}] \label{thm:T_small_d}
    Given $d \in \mathbb{N}$, there exists some constant $C = C(d) > 0$ such that for every finite, non-empty set $X \subseteq \mathbb{R}^d$ of full rank, there exist $T \subseteq X$ of size $|T| \le C$ such that
    $$
        |T + X| \ge (d+1)|X| - 5(d+1)^3.
    $$
\end{theorem}

\begin{theorem}[{\cite[Theorem 1.3]{campos24cayley}}] \label{thm:T_large_d}
    Let $d \in \mathbb{N}$ and $\gamma > d^{-1/3}$, for every finite set $X \subseteq \mathbb{R}^d$ of full rank there exists $T \subseteq X$ such that $|T| \le 4(d + 1)/\gamma$ and 
    $$
        |T + X| \ge (1 - 5\gamma) \frac{(d+1)|X|}{2}.
    $$
\end{theorem}

For convenience, we combine them in the following corollary.

\begin{cor} \label{cor:small_fingerprint}
    Given $\alpha > 0$ and sufficiently large $X \subseteq \mathbb{Z}_n$, there exists $C = C(\alpha)$ and a subset $T \subseteq X$ of size $|T| \le C \dim_f(X)$ such that
    $$
        |T + X| \ge (1 - \alpha) \frac{(\dim_f(X) + 1)|X|}{2}.
    $$
\end{cor}
\begin{proof}
    Let $d := \dim_f(X)$, and let $\phi: X \mapsto X'$ be a Freiman isomorphism where $X' \subseteq \mathbb{Z}^d$ is a set of full rank.

    \begin{itemize}
    \item If $\alpha > 5 d^{-1/3}$, then by Theorem \ref{thm:T_large_d} there exists $T' \subseteq X'$ of size $|T'| \le 40 / \alpha$ such that
    $$
        |T' + X'| \ge (1 - \alpha) \frac{(d + 1)|X'|}{2}.
    $$
    
    \item If $\alpha \le 5 d^{-1/3}$, then by Theorem \ref{thm:T_small_d} there exists $T' \subseteq X'$ of size $|T'| \le C(d)$, and as $d = O(\alpha^{-3})$ we can write $|T'| \le C(\alpha)$, such that
    $$
        |T' + X'| \ge (d+1)|X'| - 5 (d+1)^3 = (d+1)|X'| - O(\alpha^{-9}) > \frac{(d+1)|X'|}{2},
    $$
    for $X'$ sufficiently large in terms of $\alpha^{-9}$.
    \end{itemize}
    
    In both cases, the claim follows from $|X' + T'| = |\phi^{-1}(X') + \phi^{-1}(T')|$.
\end{proof}

The following proposition is our last ingredient.

\begin{proposition} \label{prop:subset_large_sum}
    Suppose $T', X$ are finite subsets of some additive group. For any integer $1 \le t \le |T'|$, there exists $T \subseteq T'$ of size $t$ such that $|T + X| \ge \frac{t}{|T'|} |T' + X|$.
\end{proposition}
\begin{proof}
    Let $a_1 \in T'$ be an element which maximises $|a + X|$, and for $i \in \{2, \ldots, t' := |T'|\}$, iteratively, choose $a_i$ to be an element $a \in T' \setminus T'_{i-1}$ which maximises $|(a + X) \setminus (T'_{i-1} + X)|$, where $T'_{i-1} = \{a_1, \ldots, a_{i-1}\}$. 
    
    Let $d_i := |(a_i + X) \setminus (T'_{i-1} + X)|$, and note that $|T' + X| = \sum_{i = 1}^{t'} d_i$. By the definition of the ordering, we have $d_i \ge d_{i+1}$ thus $d_1 + \ldots + d_t \ge \frac{t}{t'} |T' + X|$ for any $t$. Therefore, the set $T = \{a_1, \ldots, a_t\}$ satisfies the desired property.
\end{proof}

We are now ready to prove Lemma \ref{lemma:fingerprint}.

\begin{proof}[Proof of Lemma \ref{lemma:fingerprint}]
    Set $\eps = \alpha/L$, where $L = \log_{1/(1-\alpha)}(2K)$, and let $X \subseteq A$ be an $(\eps, \alpha)$-Freiman-robust subset of size $|X| > (1 - \eps L)|A|$ as guaranteed by Lemma \ref{lemma:large_robust}. In particular, we have $|X| \ge (1 - \alpha)|A|$. Let $d := \dim_f(X)$. Let $C = C(\alpha)$ be as given by Corollary \ref{cor:small_fingerprint}. We define $F \subseteq A$ in two phases.
    
    \paragraph{Phase I.} Set $X_1 = X$ and $\ell = \sqrt{|A|/d}$, and repeat the following for $i \in \{1, \ldots, \ell \}$: 
    \begin{itemize}
        \item Let $T_i \subseteq X_i$ be a set of size $|T_i| \le 2Cd$ such that
        \begin{equation} \label{eq:T_X}
            |T_i + X_i| \ge (1 - 2\alpha) \frac{(d + 1)|X_i|}{2}.
        \end{equation}
        We shall prove that such $T_i$ exists shortly.
        \item Set $X_{i+1} = X_i \setminus T_i$, and proceed to the next iteration.
    \end{itemize}
    For each $i \le \ell$ we have
    $$
        |X_i| = |X| - \sum_{j < i} |T_j| \ge |X| -2C \sqrt{d |A|} \ge \left( 1 - 2C \sqrt{2d / |X|} \right)|X|,
    $$
    where the second inequality follows from $|X| \ge (1 - \alpha)|A|$ (assuming $\alpha < 1/2$). From the upper bound on $K$ and $d < 2K$ we conclude
    $$
        \eps = \Theta_\alpha(1/\log(K)) > 4C \sqrt{K / |X|},
    $$
    thus
    $$
        |X_i| \ge (1 - \eps)|X|.
    $$
    As $X$ is $(\eps, \alpha)$-Freiman-robust, we have $\dim_f(X_i) \ge (1 - \alpha)d$. Corollary \ref{cor:small_fingerprint} gives a set $T_i' \subseteq X_i$ such that $|T_i'| \le C \dim_f(X_i)$ and
    $$
        |T_i' + X_i| \ge (1 - \alpha) \frac{(\dim_f(X_i) + 1)|X_i|}{2}.
    $$
    If $\dim_f(X_i) \le d + 1$ or $|T_i'| \le Cd$ then we can take $T_i := T_i'$. Otherwise, a subset $T_i \subseteq T_i'$ of size $C(d+1) \le 2Cd$ and which satisfies \eqref{eq:T_X} exists by Proposition \ref{prop:subset_large_sum}.
    
    For convenience, we record that \eqref{eq:T_X} and $|X_i| \ge (1 - \eps)|X| \ge (1 - 2\alpha)|A|$ imply
    \begin{equation} \label{eq:T_i_A}            
        |T_i + X| \ge |T_i + X_i| \ge (1 - 2\alpha) \frac{(d + 1)(1 - 2\alpha)|A|}{2} \ge (1 - 4\alpha) \frac{(d+1)|A|}{2}.
    \end{equation}


    \paragraph{Phase II.} The set $F_T = T_1 \cup \ldots \cup T_\ell$ forms a part of the final set $F$. The remaining of $F$ is chosen as follows. Set $Y = F_T \hat{+} F_T$ and $F' = \emptyset$, and as long as $|Y| < (1 - 5\alpha) (d + 1)|A|/2$ choose an element $x \in X$ which maximises $|(F_T + x) \setminus Y|$, add it to $F'$ and update $Y := Y \cup (F_T + x)$. Once the procedure has terminated, set $F := F_T \cup F'$. Note $|F_T + X| \ge (1 - 4\alpha)(d+1)|A|/2$, thus the procedure eventually finishes. It remains to show that $F'$ is not too large. 
    
    Let $Y_i$ denote the set $Y$ at the beggining of the $i$-th round. We claim that if $|Y_i| < (1 - 5\alpha)(d+1)|A|/2$, that is the procedure has not terminated yet, then
    \begin{equation} \label{eq:Y_increment}            
        |Y_{i+1}| \ge |Y_i| + \frac{\alpha}{2} \sqrt{d |A|}. 
    \end{equation}
    which implies that the procedure terminates after at most $2\alpha^{-1} \sqrt{d |A|}$ rounds. To this end, consider the auxiliary bipartite graph $B_i$ on vertex sets $X$ and $\mathbb{Z}_n \setminus Y_i$, where we put an edge between $x \in X$ and $y \in \mathbb{Z}_n \setminus Y_i$ if $y - x \in F_T$. By \eqref{eq:T_i_A}, for each $j \in [\ell]$ we have
    $$
        |(T_j + X) \setminus Y_i| \ge (1 - 4 \alpha) \frac{(d+1)|A|}{2} - |Y_i| > \alpha \frac{(d+1) |A|}{2}.
    $$
    As the sets $T_1, \ldots, T_\ell$ are disjoint, we conclude that the graph $B_i$ has at least $\ell \alpha d |A| / 2$ edges. For $x \in X$, the quantity $|(F_T + x) \setminus Y_i|$ is equal to the degree of $x$ in $B_i$. Therefore, for $x$ which maximises $|(F_T + x) \setminus Y_i|$ we have
    $$
        |(F_T + x) \setminus Y_i| \ge \frac{\ell \alpha d |A|}{2 |X|} = \frac{\alpha}{2} \sqrt{d |A|},
    $$
    which proves \eqref{eq:Y_increment}. The size of $F$ is at most $C' \sqrt{d|A|}$ for $C' = C(\alpha) + 2\alpha^{-1}$, thus $C' = C'(\alpha)$. By adding arbitrary elements from $X \setminus F$ to $F$, we can further guarantee that $|F| = C' \sqrt{d|A|}$.
\end{proof}

\section{Freiman-Ruzsa theorem for cyclic groups}

The following is the main result of this section.

\begin{lemma} \label{lemma:freiman_cyclic}
    Let $n \in \mathbb{N}$ be a sufficiently large prime and $A \subseteq \mathbb{Z}_n$ be a subset such that $|A+A| \le K|A|$, for sufficiently large $K$ such that $e^{O(K^{1 + o(1)})} |A| = o(n)$. Then $A$ is contained in an $\dim_f(A)$-GAP of size at most $e^{O(K^{1 + o(1)})} |A|$.
\end{lemma}

The proof gives $K^{1 + o(1)} = O(K \log^C(2K))$ for some absolute constant $C$, but for brevity we choose not to bother ourselves with lower order terms. The proof closely follows (parts of) the exposition of Green \cite{green_notes}, which is in turn based on the work of Chang \cite{chang02polynomial}. In particular, no new ideas are required. We start by collecting some definitions and known results.

Given an abelian group $G$, a \emph{$d$-dimensional coset progression} is a subset of $G$ the form $P + H$ where $P$ is a $d$-GAP and $H$ a subgroup of $G$. The size of a coset progression is defined as $\mathrm{size}(P + H) = \mathrm{size}(P)|H|$. The following version of the Freiman-Ruzsa theorem, due to Sanders \cite[Theorem 11.4]{sanders12bogolyubov}, is our starting point. Once we obtain a coset progression guaranteed by it, we use it to move from $G$ to $\mathbb{Z}^d$.

\begin{theorem} \label{thm:sanders}
    Suppose that $G$ is a discrete abelian group and $A \subset G$ is finite with $|A+A| \le K|A|$. Then there exists $d = d(K) = O(K \log^C(2K))$, for some absolute constant $C$, such that $A$ is contained in a $d$-dimensional coset progression $P  + H$ of size $\textrm{size}(P + H) \le e^{d}|A|$.
\end{theorem}

To state the next ingredient, we need some notions which further extend that of a generalised arithmetic progressions. Suppose that $B \subset \mathbb{R}^d$ is closed, centrally symmetric and convex such that $B \cap \mathbb{Z}^d$ spans $\mathbb{R}^d$ as a real vector space. For a given group homomorphism $\phi: \mathbb{Z}^d \to G$, we refer to the image $X = \phi(B \cap \mathbb{Z}^d)$ as a \emph{convex progression of dimension $d$}. The size of $X$ is $\mathrm{size}(X) = |B \cap \mathbb{Z}^d|$. For a subgroup $H$ of $G$, we call $X + H$ a \emph{convex coset progression}. By analogy with coset progressions, we define $\mathrm{size}(X + H) = \mathrm{size}(X)|H|$. We say that $X+H$ is \emph{$s$-proper}, for some $s \in \mathbb{N}$, if $\phi(x_1) - \phi(x_2) \in H$ implies $x_1 = x_2$ for all $x_1, x_2 \in sB \cap \mathbb{Z}^d$, where $sB$ denotes the dilate $\{sx \colon x \in B\}$ of $B$.

The following is due to Cwalina and Schoen \cite[Lemma 5]{cwalina13linear}.
\begin{lemma} \label{lemma:cwalina_schoen}
    Let $G$ be an abelian group, and suppose that $X + H$ is a convex coset progression of dimension $d$. Then for every $s \in \mathbb{N}$ there exists an $s$-proper convex coset progression $X' + H'$ of dimension $d' \le d$ and $\textrm{size}(X' + H') \le  s^d d^{O(d)} \textrm{size}(X + H)$, such that $X + H \subseteq X' + H'$.
\end{lemma}

Finally, to go back from convex progressions to generalised arithmetic progressions, we make use of a recent result of van Hintum and Keevash \cite{vanHintum2024john} on discrete John's theorem.

\begin{theorem} \label{thm:john}
    Suppose $B \subset \mathbb{R}^d$ is bounded, centrally symmetric and convex. Then $B \cap \mathbb{Z}^d$ is contained in a centred $d$-GAP $P$ in $\mathbb{Z}^d$ of size $O(d)^{3d} |B \cap \mathbb{Z}^d|$.
\end{theorem}

We are now ready to prove Lemma \ref{lemma:freiman_cyclic}.

\begin{proof}[Proof of Lemma \ref{lemma:freiman_cyclic}]
    Applying Theorem \ref{thm:sanders}, we obtain a $d$-dimensional coset progression $P + H \supseteq A$ in $\mathbb{Z}_n$ such that $\mathrm{size}(P)|H| \le e^d |A|$, where $d < K^{1 + o(1)}$. As $\mathbb{Z}_n$ only contains trivial subgroups, from 
    $$
        \textrm{size}(P + H) \le e^{K^{1 + o(1)}} |A| = o(n)
    $$
    we conclude $H = \{0\}$. Therefore, $P + H = P$. Suppose
    $$
        P = \left\{ a_0 + \sum_{i = 1}^{d} n_i a_i \mid -N_i \le n_i \le  N_i\right\}.
    $$
    The set $P$ can be seen as a translate of a convex progression given by the group homomorphism $\phi: \mathbb{Z}^d \to G$, $\phi((x_1, \ldots, x_d)) = \sum_{i = 1}^d x_i a_i$, and the box $B = \prod_{i = 1}^d [-N_i, N_i]$. Let $X = \phi(B \cap \mathbb{Z}^d)$. Note that $a_0 + X = P$ and $\textrm{size}(X) = \textrm{size}(P)$.
    
    Applying Lemma \ref{lemma:cwalina_schoen}, we obtain a $2$-proper convex coset progression $X' + H' \supseteq X$ of dimension $d' \le d$ and size 
    \begin{equation} \label{eq:size_B'}            
        \textrm{size}(X' + H') = \mathrm{size}(X')|H'| \le 2^d d^{O(d)}  \textrm{size}(X) \le e^{K^{1 + o(1)}} |A|.
    \end{equation}
    Again, from the assumption of the lemma we conclude $H' = \{0\}$. Let $B' \subseteq \mathbb{R}^{d'}$ be a closed, centrally symmetric and convex set and $\phi': \mathbb{Z}^d \to G$ a group homomorphism such that $X' = \phi'(B' \cap \mathbb{Z}^{d'})$. Since $X'$ is $2$-proper, the restriction $\phi'\arrowvert_{2B' \cap \mathbb{Z}^{d'}}$ is injective. Setting $A' = \phi'^{-1}(A) \cap B'$, we thus have that $\phi'\arrowvert_{A'}: A' \to A$ is a Freiman isomorphism. 

    Choose arbitrary $t \in A'$. Define the $t$-translation of $\phi'$ as $\phi'_t(x) = \phi'(x + t)$ for $x \in \mathbb{R}^{d'}$, and set $B_t' := (B' - t) - (B' - t)$. Then $B_t'$ is centrally symmetric, closed and convex. Moreover, $B_t' \subseteq 4B'$ thus 
    \begin{equation} \label{eq:B_translated}
        |B_t' \cap \mathbb{Z}^{d'}| \le 4^{d'} |B' \cap \mathbb{Z}^{d'}| \stackrel{\eqref{eq:size_B'}}{\le} e^{K^{1 + o(1)}} |A|.
    \end{equation}
    Set $A_t' = A' - t$, and note that $A_t' = \phi_t'^{-1}(A) \cap (B' - t)$. Importantly, and this was the point of this paragraph: (i) $0 \in A_t'$, and (ii) the restriction  $\phi_t' \arrowvert_{A_t'}$ remains a Freiman isomorphism between $A'$ and $A$.
    
    Let $H$ be the (proper) subspace of $\mathbb{R}^{d'}$ of smallest dimension which contains $A_t'$. Then $B_t' \cap H$ is of full rank in the lattice $\Lambda = H \cap \mathbb{Z}^{d'}$, and as $A_t'$ is Freiman-isomorphic to $A$, the dimension $d''$ of $\Lambda$ satisfies $d'' \le \dim_f(A)$ by the definition. Consider a homomorphism $\psi: H \mapsto \mathbb{R}^{d''}$ such that $\psi(\Lambda) = \mathbb{Z}^{d''}$. As $B_t' \cap H$ is centrally symmetric, closed, and convex, then so is $B'' = \psi(B_t' \cap H)$ and $|B_t' \cap \mathbb{Z}^{d'}| \ge |B'' \cap \mathbb{Z}^{d''}|$. By Theorem \ref{thm:john}, $B'' \cap \mathbb{Z}^{d''}$ is contained in a centred $d''$-GAP $P'$ in $\mathbb{Z}^{d''}$,
    $$
        P' = \left\{ \sum_{i = 1}^{d''} n_i v_i \mid -N_i' \le n_i \le N_i' \right\},
    $$
    of size $O(d'')^{3d''}|B'' \cap \mathbb{Z}^d|$. Using $d'' \le \dim_f(A) < 2K$, the upper bound on the size of $B'' \cap \mathbb{Z}^{d''}$ and \eqref{eq:B_translated}, we obtain
    $$
        O(d'')^{3d''}|B'' \cap \mathbb{Z}^d| \le  e^{K^{1 + o(1)}} |A|. 
    $$
    Finally, this gives us a centred $d''$-GAP in $\mathbb{Z}_n$:
    $$
        P'' = \left\{ \sum_{i = 1}^{d''} n_i \phi'_t(\psi^{-1}(v_i)) \mid -N_i' \le n_i \le N_i' \right\}.
    $$
    We claim that $A \subseteq a_0 + \phi'(t) + P''$. Consider some $a' \in A$ and let $a = a' - a_0$. Then there exists $b \in B'_t \cap \mathbb{Z}^{d'} \cap H$ such that $\phi_t'(b) = a$. As $b' = \psi(b)$ is contained in $P'$, we can write $b' = \sum_{i = 1}^{d''} n_i v_i$ with $-N_i' \le n_i \le N_i$ for each $i \in \{1, \ldots, d''\}$. The linear mapping $\psi$ is a bijection between $\Lambda$ and $\mathbb{Z}^{d''}$, thus
    $$
      b = \sum_{i = 1}^{d''} n_i \psi^{-1}(v_i).  
    $$
    Finally, using the fact that $\phi_t'$ is an affine mapping, we get
    $$
        a' = a_0 + a = a_0 + \phi_t'(b) = a_0 + \phi(t) + \sum_{i = 1}^{d''} n_i \phi'(\psi^{-1}(v_i)).
    $$
\end{proof}

\section{Proof of Theorem \ref{thm:main}}

We follow the proof of Campos, Dahia, and Marciano \cite{campos24cayley}, which is in turn based on the approach of Green \cite{gree05clique}. Lemma \ref{lemma:fingerprint} and Lemma \ref{lemma:freiman_cyclic} are two new ingredients.

\begin{proof}[Proof of Theorem \ref{thm:main}]
    Throughout the proof we set $k := (2 + 4\delta) \log_{1/(1-p)}(n)$. We frequently use 
    $$
        1 - p = n^{-(2 + 4\delta)k}.
    $$
    
    A subset $A \subseteq \mathbb{Z}$ is an independent set in $\Gamma_S$ if $A \hat{+} A  \subseteq S^c$, where $S^c = \mathbb{Z}_n \setminus S$ is the complement of $S$ and $A \hat{+} A = \{ a + a' \colon a, a' \in A, a \neq a'\}$. Therefore, our goal is to show
    $$
        \Pr_{S} \left[ \exists A \in \binom{\mathbb{Z}_n}{k} \colon A \hat{+} A \subseteq S^c \right] \to 0,
    $$
    where $S$ is a $p$-random subset of $\mathbb{Z}_n$. To make this explicit, throughout the proof we use $S_p$ to signify that $S \subseteq \mathbb{Z}_n$ is a $p$-random subset.
    
    Let $\sigma(A) := |A+A| / |A|$ denote the \emph{doubling constant} of $A$. Following \cite{campos24cayley}, we split $k$-subsets of $\mathbb{Z}_n$ into the following three groups:
    \begin{align*}
        \mathcal{X}_1 &= \left\{ A \in \binom{\mathbb{Z}_n}{k} \colon \sigma(A) < k^{1/4} \right\} \\
        \mathcal{X}_2 &= \left\{ A \in \binom{\mathbb{Z}_n}{k} \colon k^{1/4} \le \sigma(A) < \delta k/10 \right\} \\
        \mathcal{X}_3 &= \left\{ A \in \binom{\mathbb{Z}_n}{k} \colon \delta k / 10 \le \sigma(A) \right\}.
    \end{align*}
    Dealing with $\mathcal{X}_2$ and $\mathcal{X}_3$ is identical to \cite{campos24cayley,gree05clique}. We spell out the details for the sake of completeness. The main contribution of \cite{campos24cayley} is a more efficient treatment of the sets in $\mathcal{X}_1$, and it is this part that we improve.

    \paragraph{Sets with small doubling constant.} Consider some $A \in \mathcal{X}_1$. Let $F \subseteq X \subseteq A$ be subsets as given by Lemma \ref{lemma:fingerprint}. Recall that $|X| \ge (1 - \alpha)|A| \ge |A|/2$, and
    \begin{equation} \label{eq:F}
        |F| = C \sqrt{d k} \; \text{ and } \;|F \hat{+} F| \ge (1 - \alpha) \frac{(d+1)k}{2} - |F| \ge (1 - 2\alpha)\frac{(d+1)k}{2},
    \end{equation}
    for some $C = C(\alpha)$ and $d = \dim_f(X)$. From $|X| \ge |A|/2$ we conlcude $\sigma(X) \le 2\sigma(A) < 2k^{1/4}$. By Lemma \ref{lemma:freiman_cyclic} applied with $X$, there exists a $d$-GAP $P \supset X$ of size at most $\exp(k^{1/4 + o(1)})$ (we use $o(1)$ in the exponent and the fact that $k$ is sufficiently large to subsume all lower order terms). We can further assume that all the parameters $N_1, \ldots, N_{\dim_f(X)}$ of $P$ are powers of $2$, and that the size of $P$ is exactly $\exp(k^{1/4 + o(1)})$. Let us denote with $\mathcal{A}_d$ the family of all such $d$-GAPS, and $\mathcal{F}(P)$ all subsets $F \subseteq P$ satisfying \eqref{eq:F}. Finally, note that $\Pr[F \hat{+} F \subseteq S_p^c] = (1 - p)^{|F \hat{+} F|}$.
    
    Based on preceding observations, we arrive at the following union bound:
    \begin{align*}
        \Pr\left[ \exists A \in \mathcal{X}_1 \colon A \hat{+} A \subseteq S_p^c \right] &\le 
        \sum_{d = 1}^{4k^{1/4}} \sum_{P \in \mathcal{A}_d} \sum_{F \in \mathcal{F}(P)} \Pr[F \hat{+} F \subseteq S_p^c] \\
        &\le \sum_{d = 1}^{4k^{1/4}} n^{d + 1} \binom{k^{1/4 + o(1)}}{d - 1} \binom{e^{k^{1/4 + o(1)}}}{C (d k)^{1/2}} (1 - p)^{(1 - 2\alpha) (d+1) k / 2}.
    \end{align*}
    The first two terms count the number of $d$-GAPs of size $\exp(k^{1/4 + o(1)})$, as well as all the parameters $N_1, \ldots, N_d$,  being a power of two. Recalling $k = (2 + \delta) \log_{1/(1-p)}(n)$, this further simplifies to
    $$
        \sum_{d_f = 1}^{4 k^{1/4}} n^{d+1} e^{k^{3/4 + o(1)} \sqrt{d}} n^{-(1 - 3\alpha)(1 + \delta) d} e^{- \alpha p d k / 2}.
    $$
    As $p d k \gg k^{3/4 + o(1)} \sqrt{d}$ for $p \ge (\log n)^{-1/3 + o(1)}$ (with the latter $o(1)$ decaying significantly slower), and $(1 - 3\alpha)(1 + \delta) > 1$, we conclude that the previous sum vanishes.


    \paragraph{Sets with large doubling constant.}
    By \cite[Proposition 23]{gree05clique} (see \cite[Section 8.2]{campos24cayley} for details), there are at most $n^{(2 + 2\delta)m/k} k^{4k}$ sets $A \in \mathcal{X}_2$ such that $|A \hat{+} A| = m$. We thus get the following union bound:
    \begin{align*}
        \Pr[\exists A \in \mathcal{X}_2 \colon A \hat{+} A \subseteq S_p^c] &\le \sum_{m = k^{1 + 1/4}}^{\delta k / 10} n^{(2 + 2\delta)m/k} k^{4k} (1 - p)^{(1 - \delta/2)m + \delta m / 2} \\
        &\le \sum_{m = k^{1 + 1/4}}^{\delta k /10} n^{(2 + 2\delta)m/k} e^{4k \log(k)} n^{-(1 - \delta)(2+4\delta)m/k} e^{-\delta p m / 2}.
    \end{align*}
    As $p k^{4/3} \gg k \log(k)$ for $p \ge (\log n)^{-1/3 + o(1)}$, the sum vanishes.

    \paragraph{Sets with linear doubling constant.} By \cite[Proposition 5.1]{green16counting} (see \cite[Section 8.3]{campos24cayley} for details), there are at most $n^{(2 + 2\delta)m/k}$ sets $A \in \mathcal{X}_3$ such that $|A \hat{+} A| = m$. Therefore,
    $$
        \Pr[\exists A \in \mathcal{X}_2 \colon A \hat{+} A \subseteq S_p^c] \le \sum_{m = \delta k / 10}^k n^{(2 + 2\delta)m/k} (1 - p)^m = \sum_{m = \delta k / 10}^k n^{(2 + 2\delta)m/k} n^{-(2 + 4\delta)m/k} \to 0.
    $$
\end{proof}

\paragraph{Acknowledgment.} The author thanks Gabriel Dahia and Jo\~{a}o Pedro Marciano for timely and valuable comments.

\bibliographystyle{abbrv}
\bibliography{references}

\end{document}